\documentclass[11pt]{amsart}
\usepackage{amsmath}
\usepackage{amssymb}
\usepackage{amscd}
\usepackage[dvips]{epsfig}
\usepackage{graphics}

\theoremstyle{plain}
\newtheorem{theorem}{Theorem}[section]
\newtheorem{lemma}[theorem]{Lemma}

\newtheorem{corollary}[theorem]{Corollary}

\theoremstyle{definition}

\newtheorem{remark}[theorem]{Remark}
\theoremstyle{remark}

\begin{document}
\def \CPb {\overline{\mathbf{CP}}^{\,2}}
\def \CP {{\mathbf{CP}}^{\,2}}
\def \Sig{\Sigma}

\def \a {\alpha}
\def \b {\beta}
\def \G {\Gamma}

\def \oa {\overline{a}}
\def \ob {\overline{b}}
\def \og {\overline{\gamma}}

 \title{On Sections of genus two Lefschetz fibrations}

 \author{S\.{i}nem \c{C}el\.{i}k Onaran}

 \address{Department of Mathematics, Middle East Technical University, Ankara, Turkey and School of Mathematics, Georgia Institute of Technology, Atlanta, Georgia.} \email{sinemo@math.gatech.edu, e114485@metu.edu.tr}
  \subjclass{57M50, 57R17}
\keywords{Lefschetz fibrations, sections of Lefschetz fibrations}

\begin{abstract} In this note we find new relations in the mapping class group of a genus $2$ surface with $n$ boundary components for $n = 1,\ldots, 8$ which induce a genus $2$ Lefschetz fibration $\CP\#13\CPb \rightarrow S^{2}$ with $n$ disjoint sections. As a consequence, we show any holomorphic genus $2$ Lefschetz fibration without separating singular fibers admits a section.
\end{abstract}
 \maketitle

 \setcounter{section}{0}

\section{Introduction} 

\par The study of Lefschetz fibrations is important in low dimensional topology because of a close relationship between symplectic $4$-manifolds and Lefschetz fibrations, \cite{Don}, \cite{GS}. Also, sections of Lefschetz fibrations play an important role in the theory. For example, in the presence of a section the fundamental group and the signature of a Lefschetz fibration can be easily computed.

\par This paper provides sections for genus $2$ Lefschetz fibration $\CP\#13\CPb \rightarrow S^{2}$ with global monodromy given by the relation $(t_{c_1}t_{c_2}t_{c_3}t_{c_4}t_{c_5}^2t_{c_4}t_{c_3}t_{c_2}t_{c_1})^2=1$ in the mapping class group $\G_{2}$ of a closed genus $2$ surface where each $c_i$ is a simple closed curve as in Figure ~ \ref{f3} and $t_{c_i}$ is a right handed Dehn twist about $c_i$, $i = 1, \ldots, 5$. We want to remark that in \cite{KO}, similar relations found in mapping class group $\G_{1, n}$ of a genus $1$ surface with $n$ boundary components for $n = 4,\ldots, 9$ giving an elliptic Lefschetz fibration $\CP\#9\CPb \rightarrow S^{2}$ with $n$ disjoint sections.
 \par In the following section, we recall definitions and some relations in the mapping class group used in our computations. We fix notations to be used throughout the paper. In Section 3, we give brief background information on Lefschetz fibrations. In Section 4, we provide the necessary relations in the mapping class group $\G_{2,n}$ for a genus $2$ Lefschetz fibration $\CP\#13\CPb \rightarrow S^{2}$ with $n$ disjoint sections, $n = 1,\ldots, 6$. Finally, in Section 5 we list several observations and open problems related to the sections of Lefschetz fibrations. We show that a genus $2$ Lefschetz fibration $\CP\#13\CPb \rightarrow S^{2}$ may admit at most $12$ disjoint sections. We provide relations in the corresponding mapping class group that give $n = 7, 8$ disjoint sections for genus $2$ Lefschetz fibration $\CP\#13\CPb \rightarrow S^{2}$. We conclude that any holomorphic genus $2$ Lefschetz fibration without separating singular fibers admits a section.

\section{Mapping Class Groups} 
\par Let $\Sigma_{g, n}^k$ denote oriented, connected genus $g$ surface with $n$ boundary components and $k$ marked points. \emph{The mapping class group} of $\Sigma_{g, n}^k$ is defined as the isotopy classes of orientation preserving self diffeomorphisms of $\Sigma_{g, n}^k$ which are assumed to fix the marked points and the points on the boundary. Denote the mapping class group of $\Sigma_{g, n}^k$ by $\G_{g, n}^k$. When $k=0$, denote the mapping class group of $\Sigma_{g, n}$ by $\G_{g, n}$. 

\par Let $a$ be a simple closed curve on $\Sigma_{g, n}^k$. A \emph{right handed Dehn twist} $t_a$ about $a$ is the isotopy class of a self diffeomophism of  $\Sigma_{g, n}^k$ obtained by cutting $\Sigma_{g, n}^k$ along $a$ and gluing back after twisting one side by $2\pi$ to the right. The inverse of a right handed Dehn twist is a \emph{left handed Dehn twist}, denoted by $t_a^{-1}$.

Let us now briefly mention the facts and the relations that we will use in our computations. For the proofs see \cite{BM} and \cite{I}. Let $a$ be a simple closed curve on $\Sigma_{g, n}^k$. If $f$: $\Sigma_{g, n}^k \rightarrow \Sigma_{g, n}^k$ is an orientation preserving diffeomophism, then $ft_af^{-1} = t_{f(a)}$.

\par For simplicity in the rest of the paper we denote a right handed Dehn twist $t_a$ along $a$ by $a$ and a left handed Dehn twist $t_a^{-1}$ by $\oa$. The product $ab$ means we first apply the Dehn twist $b$ then the Dehn twist $a$. We also remark that by a boundary curve in a given surface with boundary we mean a simple closed curve parallel to the boundary component of the given surface.
\par The following relations will be useful in our computations:
\par If $a$ and $b$ are two disjoint simple closed curves on  $\Sigma_{g, n}^k$, then we have the \emph{commutativity relation}. Dehn twists along $a$ and $b$ commute; $ab = ba$.
\par If $a$ and $b$ are two simple closed curves on $\Sigma_{g, n}^k$ intersecting transversely at a single point, then the corresponding Dehn twists satisfy the \emph{braid relation}; $aba = bab$.
\par Consider a sphere with four holes and the boundary curves $\delta_1, \delta_2, \delta_3, \delta_4$ and the simple closed curves $\a, \gamma,\sigma$ as shown in Figure ~ \ref{f1}. Then, there is a relation first discovered by Dehn and then rediscovered and named as \emph{lantern relation} by D. Johnson:
\begin{equation*}
\delta_1\delta_2\delta_3\delta_4 = \gamma\sigma\a.
\end{equation*}
\par Let $\Sigma_{1, 3}$ be a torus with three boundary curves $\delta_1, \delta_2, \delta_3$, then in $\G_{1, 3}$ we have the following relation called \emph{star relation} in \cite{G}. For the simple closed curves $a_1, a_2, a_3, b$, see Figure ~ \ref{f1}.
\begin{equation*}
\delta_1\delta_2\delta_3 = (a_1a_2a_3b)^3.
\end{equation*}
\begin{figure}[h!]
 \begin{center}
    \includegraphics[width=13cm]{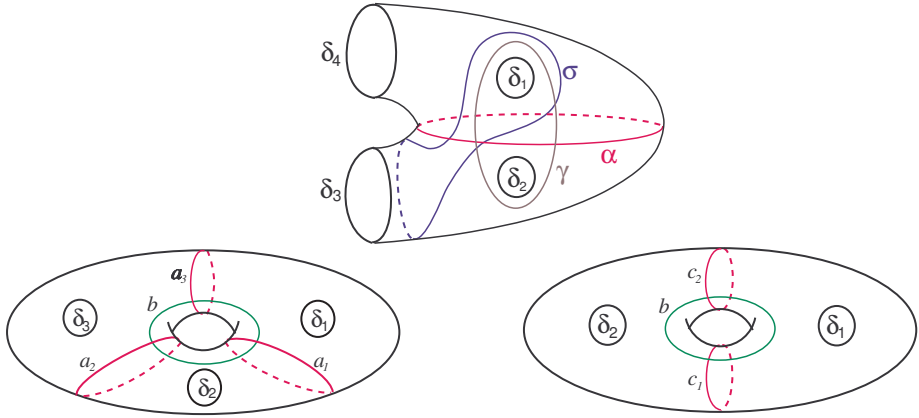}
  \caption{The Lantern relation, $\delta_1\delta_2\delta_3\delta_4 = \gamma\sigma\a$. The Star relation, $\delta_1\delta_2\delta_3 = (a_1a_2a_3b)^3$ and The two-holed torus relation, $\delta_1\delta_2 = (c_1bc_2)^4$.}
  \label{f1}
   \end{center}
 \end{figure}
\par Consider a torus $\Sigma_{1, 2}$ with two boundary curves $\delta_1, \delta_2$ and the simple closed curves $c_1, c_2, b$ as shown in Figure~ \ref{f1}. Then we have the following \emph{chain relation} for two holed torus:
\begin{equation*}
\delta_1\delta_2 = (c_1bc_2)^4.
\end{equation*}
\par We also need the \emph{chain relations} for genus $2$ case. If $c_1, c_2, c_3, c_4, c_5$ are the chain of curves as shown in Figure ~ \ref{f2}, then for 
a genus $2$ surface $\Sigma_{2, 1}$ with one boundary curve $\delta_1$, we have
\begin{equation*}
\delta_1 = (c_1c_2c_3c_4)^{10}
\end{equation*}
and for a genus $2$ surface $\Sigma_{2, 2}$ with two boundary curves $\delta_1, \delta_2$, we have 
\begin{equation*}
\delta_1\delta_2 = (c_1c_2c_3c_4c_5)^{6}.
\end{equation*}
\begin{figure}[hbt]
\begin{center}
    \includegraphics[width=14cm]{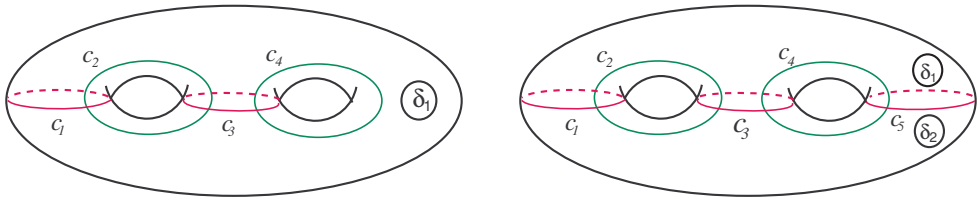}
  \caption{The chain relations: $\delta_1 = (c_1c_2c_3c_4)^{10}$ and $\delta_1\delta_2 = (c_1c_2c_3c_4c_5)^{6}$.}
  \label{f2}
\end{center}
 \end{figure}
\section{Lefschetz Fibrations}
\par A \emph{Lefschetz fibration} on a closed, connected, oriented smooth four manifold $X$ is a map $ f: X  \rightarrow   \Sig$, where $\Sigma$ is a closed, connected, oriented smooth surface, such that $f$ is surjective, $f$ has isolated critical points and for each critical point $p$ of $f$ there is an orientation preserving local complex coordinate chart on which $f$ takes the form $f(z_{1}, z_{2}) = {z_{1}^{2}} + {z_{2}^{2}}$.   
\par The Lefschetz fibration $f$ is a smooth fiber bundle away from critical points. A regular fiber of $f$ is diffeomorphic to a closed, oriented smooth genus $g$ surface for some $g$ and we define the \emph{genus} of the Lefschetz fibration as the genus of the regular fiber. 
\par A \emph{singular fiber} is a fiber containing a critical point. We assume that each singular fiber contains only one critical point. The singular fiber can be obtained by taking a simple closed curve on a regular fiber and shrinking it to a point. This simple closed curve describing the singular fiber is called the \emph{vanishing cycle}. If this curve is a nonseparating curve, then the singular fiber is called \emph{nonseparating}, otherwise it is called \emph{separating}. For a genus $g$ Lefschetz fibration over $S^{2}$, the product of Dehn twists along the vanishing cycles gives us the \emph{global monodromy} of the Lefschetz fibration.
\par The right handed Dehn twists $c_i$ along the simple closed curves $c_i$, $i = 1, \ldots, s$, on a closed surface $\Sigma_{g}$ with the relation 
\begin{equation*}
c_1c_2\cdots c_s = 1
\end{equation*}
defines a genus $g$ Lefschetz fibration over $S^{2}$ with the vanishing cycles $c_1, \ldots, c_s$. In particular, in $\G_{2}$ the following relations hold: 
\begin{eqnarray*}
(c_1c_2c_3c_4c_5^2c_4c_3c_2c_1)^2 &=& 1,\\
(c_1c_2c_3c_4c_5)^6 &=& 1,\\
(c_1c_2c_3c_4)^{10} &=& 1
\end{eqnarray*}
where $c_1,\ldots, c_5$ are the simple closed curves as in Figure ~ \ref{f3}. For each relation above we have genus $2$ Lefschetz fibrations over $S^{2}$  with total spaces $\CP\#13\CPb$, $K3\#2\CPb$ and the Horikawa surface $H$, respectively. For details on Lefschetz fibrations see \cite{A}, \cite{GS}.
\begin{figure}[hbt]
\begin{center}
    \includegraphics[width=11cm]{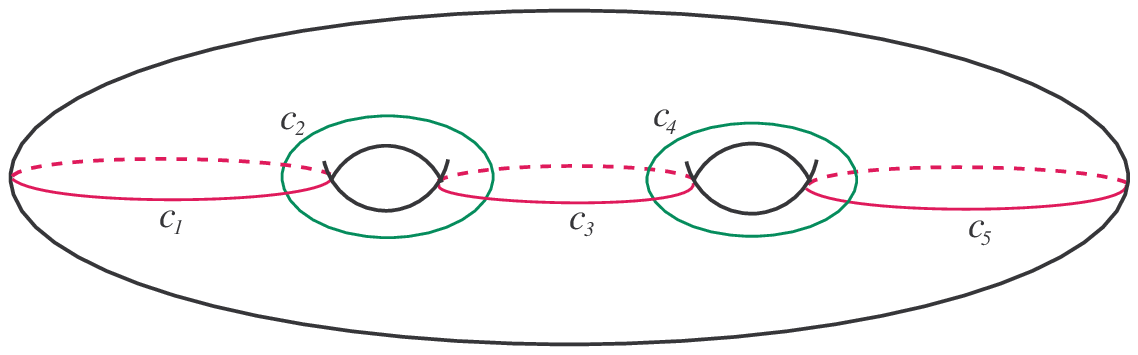}
  \caption{$\Sigma_{2}$}
  \label{f3}
\end{center}
 \end{figure}
\par A section of a Lefschetz fibration is a map $\sigma:\Sigma\rightarrow X$ such that $f\sigma = id_{\Sigma}$. Consider a collection of simple closed curves $c_1, \ldots, c_s$ on a genus $g$ surface $\Sigma_{g, n}$ with the following relation in $\G_{g,n}$
\begin{equation*}
c_1\cdots c_s = \delta_1\cdots \delta_n
\end{equation*}
where $\delta_1,\ldots, \delta_n$ are the boundary curves. This relation defines a genus $g$ Lefschetz fibration over $S^{2}$ admitting $n$ disjoint sections with the global monodromy $c_1\cdots c_s = 1$. To see this, note that after gluing a disk along each boundary curve one gets the relation $c_1\cdots c_s = 1$ in $\G_{g}$. Thus, this relation will give us a genus $g$ Lefschetz fibration over $S^{2}$ as before. One can use the centers of the capping disks to construct sections.
\par In the following sections, we find relations of above type in the mapping class group $\G_{2,n}$ for $n = 1,\ldots, 8$. These relations give us a genus $2$ Lefschetz fibration \;$\CP\#13\CPb \rightarrow S^{2}$ with $n$ disjoint sections for $n = 1,\ldots, 8$.

\section{Relations in $\G_{2,n}$}
For each $n = 1,\ldots, 6$, we write the product of right handed Dehn twists along the boundary curves $\delta_1,\ldots, \delta_n$ as a product of twenty right handed Dehn twists along non-boundary parallel simple closed curves on a genus $2$ surface $\Sigma_{2,n}$ with $n$ holes. Namely, we provide relations of the form $\delta_1 \cdots \delta_n = \beta_1 \cdots \b_{20}$, where $\beta_1, \ldots, \beta_{20}$ are non-boundary parallel simple closed curves on $\Sigma_{2,n}$. After gluing disks along the boundary curves $\delta_1,\ldots, \delta_n$, we get the relation $1 = \b_1\cdots\b_{20}$ in $\G_{2}$. By using the commutativity relation and the braid relation, one can simplify the right hand side of the equation $ 1 = \beta_1 \cdots \b_{20}$ so that it gives us the global monodromy $(c_1c_2c_3c_4c_5^2c_4c_3c_2c_1)^2 = 1$ of a Lefschetz fibration $\CP\#13\CPb \rightarrow S^{2}$ where $c_1, \ldots, c_5$ are simple closed curves as in Figure ~ \ref{f3}. Thus, these relations we find in $\G_{2,n}$ for $n = 1,\ldots, 6$ in the following subsections give us a genus $2$ Lefschetz fibration $\CP\#13\CPb \rightarrow S^{2}$ admitting $n$ disjoint sections.
\subsection{Genus two surface with one hole} Consider the genus $2$ surface $\Sigma_{2,1}$ with one boundary curve $\delta_1$ as in Figure ~ \ref{f4}. We have
\begin{equation*}
\delta_1 = (a_1b_1a_2b_2)^{10} = (a_1b_1a_2b_2)^5(a_1b_1a_2b_2)^5.
\end{equation*}
Using the commutativity relation and the braid relation, one can show
\begin{equation*}
(a_1b_1a_2b_2)^5 = (a_1b_1a_2)^4(b_2a_2b_1a_1^2b_1a_2b_2).
\end{equation*}
Notice the embedded two holed torus with two boundary curves $a_3, a_4$ in $\Sigma_{2,1}$. Then, by the chain relation for torus, we have $(a_1b_1a_2)^4 = a_3a_4$.
\par By combining the above relations, we get
\begin{equation}
\begin{array}{ccl}
\delta_1 &=& (a_1b_1a_2)^4(b_2a_2b_1a_1^2b_1a_2b_2)(a_1b_1a_2)^4(b_2a_2b_1a_1^2b_1a_2b_2) \nonumber \\ &=&  a_3a_4(b_2a_2b_1a_1^2b_1a_2b_2)a_3a_4(b_2a_2b_1a_1^2b_1a_2b_2) \\ &=& 
(a_3a_4b_2a_2b_1a_1^2b_1a_2b_2)^2.
\end{array}
\end{equation}
\begin{figure}[hbt]
\begin{center}
    \includegraphics[width=13cm]{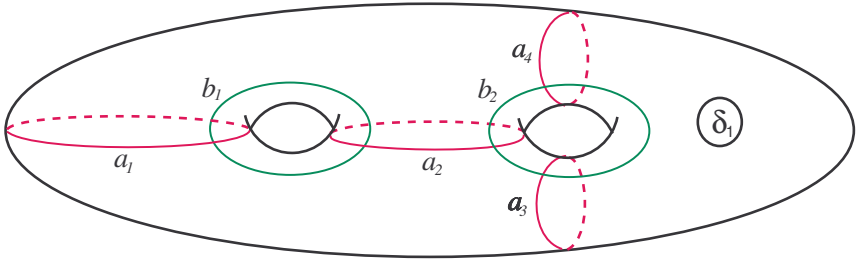}
  \caption{$\Sigma_{2,1}$}
  \label{f4}
\end{center}
 \end{figure}
\subsection{Genus two surface with two holes} Consider the genus $2$ surface $\Sigma_{2,2}$ with two boundary curves $\delta_1, \delta_2$ in Figure 5 and notice the embedded sphere with four boundary curves $\delta_1, \delta_2, a_3, a_4$ in $\Sigma_{2,2}$. Then by the lantern relation, we have
\begin{equation*}
a_3a_4\delta_1\delta_2 = \gamma\sigma a_5.
\end{equation*}
Notice the genus $2$ surface with one boundary curve $\gamma$ in the Figure ~ \ref{f5}. Thus, we have the chain relation $\gamma = (a_1b_1a_2b_2)^{10}$. Now, replacing $\gamma$ in the lantern relation and then replacing the two-holed torus relation $(a_1b_1a_2)^4 = a_3a_4$ in the equation, we get
\begin{equation}
\begin{array}{ccl}
a_3a_4\delta_1\delta_2 &=&  \gamma\sigma a_5 \nonumber \\ &=&  (a_1b_1a_2b_2)^{10}\sigma a_5  \\ &=& (a_1b_1a_2b_2)^5(a_1b_1a_2b_2)^5\sigma a_5  \\ &=& 
(a_1b_1a_2)^4(b_2a_2b_1a_1^2b_1a_2b_2)(a_1b_1a_2)^4(b_2a_2b_1a_1^2b_1a_2b_2) \sigma a_5 \nonumber \\ &=& a_3a_4(b_2a_2b_1a_1^2b_1a_2b_2)a_3a_4(b_2a_2b_1a_1^2b_1a_2b_2) \sigma a_5.
\end{array}
\end{equation}
We simplify the equation as
\begin{equation*}
\delta_1\delta_2 = b_2a_2b_1a_1^2b_1a_2b_2a_3a_4(b_2a_2b_1a_1^2b_1a_2b_2) \sigma a_5.
\end{equation*}
\begin{figure}[hbt]
\begin{center}
    \includegraphics[width=18cm]{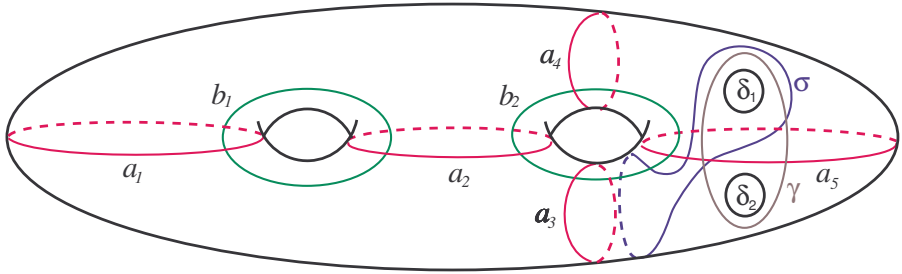}
  \caption{$\Sigma_{2,2}$}
  \label{f5}
\end{center}
 \end{figure}
\subsection{Genus two surface with three holes} First observe the lantern relation for the sphere with four boundary curves $\delta_1, \delta_2, a_4, a_5$ in 
$\Sigma_{2,3}$ in Figure ~ \ref{f6}:
\begin{equation*}
a_4a_5\delta_1\delta_2 = \gamma\sigma a_6.
\end{equation*}
For the three holed torus with the boundary curves $\gamma, a_1, a_2$,  we have the star relation $\gamma a_1a_2 = (a_4a_5a_3b_2)^3$, and for the three holed torus with the boundary curves $\delta_3, a_4, a_5$,  we have $\delta_3a_4a_5 = (a_1a_2a_3b_1)^3.$
\par Now combine the relations. Plug in $\delta_3, \gamma$. Then, simplify by using the commutativity relation and the braid relation. Note that all $a_i$'s commute for $i=1, \ldots, 6$. The simple closed curves $ a_1, a_2, a_3$ intersect $b_1$ transversely at a single point and $a_4, a_5, a_3$ intersect $b_2$ transversely at a single point.
\begin{equation*}
\delta_1\delta_2 = \oa_4\oa_5\gamma\sigma a_6
\end{equation*}
\begin{equation}
\begin{array}{ccl}
\delta_1\delta_2\delta_3 &=& \delta_3\oa_4\oa_5\gamma\sigma a_6 \nonumber \\ &=& \delta_3 \oa_4\oa_5\oa_1\oa_2(a_4a_5a_3b_2)^3\sigma a_6 \nonumber \\ &=& \oa_1\oa_2(\delta_3)\oa_4\oa_5(a_4a_5a_3b_2)(a_4a_5a_3b_2)^2\sigma a_6 \nonumber \\ &=& \oa_1\oa_2((a_1a_2a_3b_1)^3\oa_4\oa_5)a_3b_2(a_4a_5a_3b_2)^2\sigma a_6 \nonumber \\ &=& a_3b_1(a_1a_2a_3b_1)^2a_3(\oa_4\oa_5b_2a_4a_5)a_3b_2 a_4a_5a_3b_2\sigma a_6\nonumber \\ &=& a_3b_1(a_1a_2a_3b_1)^2a_3\beta a_3b_2a_4a_5a_3b_2\sigma a_6
\end{array}
\end{equation}
where $\beta = \oa_5\oa_4b_2a_4a_5$.
\begin{figure}[hbt]
\begin{center}
    \includegraphics[width=16cm]{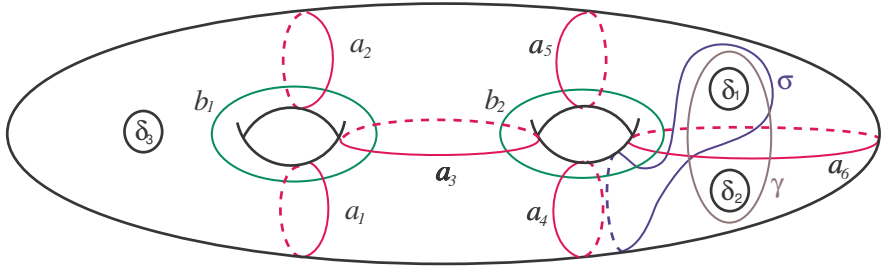}
  \caption{$\Sigma_{2,3}$}
  \label{f6}
\end{center}
 \end{figure}
\subsection{Genus two surface with four holes} Here, we will use the three-holed genus two relation we found in the previous Section 4.3. Notice the genus $2$ surface with three boundary curves $\delta_3, \delta_4, \gamma$ in Figure ~ \ref{f7}. Then,
\begin{equation}
\begin{array}{ccl}
\delta_3\delta_4\gamma &=& a_3b_2(a_5a_4a_3b_2)^2a_3\beta_1 a_3b_1a_2a_1a_3b_1\sigma_1a_7 \nonumber \\ &=& a_3\beta_1 a_3b_1a_2a_1a_3b_1\sigma_1a_7a_3b_2(a_5a_4a_3b_2)^2.
\end{array}
\end{equation}
where $\beta_1 = \oa_1\oa_2b_1a_2a_1$.
Note that, we identify the curves $(a_1, a_2, a_3, a_4, a_5, a_6)$ in $\Sigma_{2,3}$ in Figure ~ \ref{f6} with the curves $(a_5, a_4, a_3, a_2, a_1, a_7)$ in $\Sigma_{2,4}$ in Figure ~ \ref{f7}.
\begin{figure}[hbt]
\begin{center}
    \includegraphics[width=16cm]{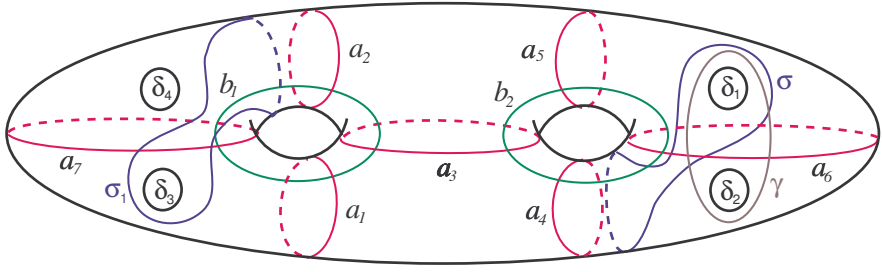}
  \caption{$\Sigma_{2,4}$}
  \label{f7}
\end{center}
 \end{figure}
\\By the lantern relation for the sphere with four boundary curves $\delta_1, \delta_2, a_4, a_5$, we have
\begin{equation*}
a_4 a_5\delta_1\delta_2= \gamma\sigma a_6.
\end{equation*}
Now, combine the above relations,
\begin{equation}
\begin{array}{ccl}
\delta_1\delta_2\delta_3 \delta_4 &=& a_3\beta_1 a_3b_1a_2a_1a_3b_1\sigma_1a_7a_3b_2(a_5a_4a_3b_2)^2\og 
\overline{a}_4\oa_5\gamma\sigma a_6 \nonumber \\ &=& a_3\beta_1 a_3b_1a_2a_1a_3b_1\sigma_1a_7a_3b_2(a_5a_4a_3b_2)(a_3a_5a_4b_2)\oa_4\oa_5\sigma a_6 \nonumber \\ &=& a_3\beta_1 a_3b_1a_2a_1a_3b_1\sigma_1a_7a_3b_2(a_5a_4a_3b_2) a_3\beta_2\sigma a_6.
\end{array}
\end{equation}
where $\beta_1 =\oa_1\oa_2b_1a_2a_1$ and $\beta_2 = a_5a_4b_2\oa_4\oa_5$.
\subsection{Genus two surface with five holes}
The lantern relation for the sphere with four boundary curves $\delta_1, \delta_2, a_4, a_5$ in $\Sigma_{2,5}$ in Figure ~ \ref{f8} is
\begin{equation*}
a_4 a_5\delta_1\delta_2= \gamma\sigma a_6.
\end{equation*}
Notice the genus $2$ surface with four boundary curves $\delta_3, \delta_4, \delta_5, \gamma $ in Figure 8. Identify the curves $(\delta_1, \delta_2, a_5, a_6, \sigma)$ in $\Sigma_{2,4}$ in Figure ~ \ref{f7} with the curves $(\delta_5, \gamma, a_8, a_5, \sigma_2)$ in $\Sigma_{2,5}$ in Figure ~ \ref{f8}. Then, by the relation given in Section 4.4 it follows that 
\begin{equation*}
\delta_3\delta_4\delta_5\gamma = a_3\beta_1 a_3b_1a_2a_1a_3b_1\sigma_1a_7a_3b_2(a_8a_4a_3b_2) a_3\beta_2\sigma_2 a_5.
\end{equation*}
where $\beta_1 = \oa_1\oa_2b_1a_2a_1$, $\beta_2 = a_8a_4b_2\oa_4\oa_8$.
\par Now, combine the above relations and simplify the equation as
\begin{equation}
\begin{array}{ccl}
\delta_1\delta_2\delta_3\delta_4\delta_5 &=& a_3\beta_1 a_3b_1a_2a_1a_3b_1\sigma_1a_7a_3b_2(a_8a_4a_3b_2) a_3\beta_2\sigma_2 a_5\overline{\gamma}            \overline{a}_4\oa_5\gamma\sigma a_6 \nonumber \\ &=& a_3\beta_1 a_3b_1a_2a_1a_3b_1\sigma_1a_7a_3b_2(a_8a_4a_3b_2) a_3\beta_2\sigma_2 \oa_4\sigma a_6 \nonumber \\ &=& a_3\beta_1 a_3b_1a_2a_1a_3b_1\sigma_1a_7a_3(\oa_4)b_2(a_4a_8a_3b_2) a_3\beta_2\sigma_2 \sigma a_6 \nonumber \\ &=& a_3\beta_1 a_3b_1a_2a_1a_3b_1\sigma_1a_7a_3\beta_3(a_8a_3b_2) a_3\beta_2\sigma_2 \sigma a_6 

\end{array}
\end{equation}
where $\beta_1 = \oa_1\oa_2b_1a_1a_2$, $\beta_2 = a_8a_4b_2\oa_4\oa_8$ and $\beta_3 = \oa_4b_2a_4$.
\begin{figure}[hbt]
\begin{center}
    \includegraphics[width=17cm]{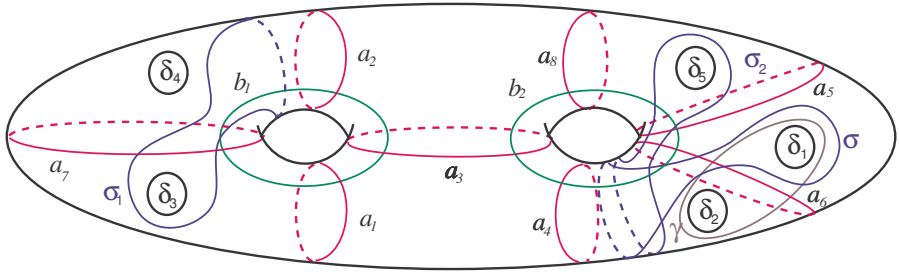}
  \caption{$\Sigma_{2,5}$}
  \label{f8}
\end{center}
 \end{figure}
\subsection{Genus two surface with six holes}
The lantern relation for the sphere with four boundary curves $\delta_1, \delta_2, a_4, a_5$ in $\Sigma_{2,6}$ in Figure ~ \ref{f9} is
\begin{equation*}
a_4 a_5\delta_1\delta_2 = \gamma\sigma a_6.
\end{equation*}
Now, identify the curves $(\delta_1, \delta_2, a_6, a_8, \sigma, \sigma_2)$ in $\Sigma_{2,5}$ in Figure ~ \ref{f8} with the curves $(\delta_6, \gamma, a_5, a_9, \sigma_2, \sigma_3)$ in $\Sigma_{2,6}$ in Figure ~ \ref{f9}. Then, by the relation given in Section 4.5 for the genus $2$ surface with five boundary curves $\delta_3, \delta_4, \delta_5, \delta_6, \gamma $, we have
\begin{equation*}
\delta_6\gamma\delta_3\delta_4\delta_5 = a_3\beta_1 a_3b_1a_2a_1a_3b_1\sigma_1a_7a_3\beta_3a_9a_3b_2 a_3\beta_2\sigma_3 \sigma_2 a_5.
\end{equation*}
where $\beta_1 =\oa_1\oa_2b_1a_2a_1$, $\beta_2 = a_9a_4b_2\oa_4\oa_9$, $\beta_3 = \oa_4b_2a_4$. 
\par Now combine the above relations, we get
\begin{equation}
\begin{array}{ccl}
\delta_1\delta_2\delta_3\delta_4\delta_5\delta_6 &=& a_3\beta_1 a_3b_1a_2a_1a_3b_1\sigma_1a_7a_3\beta_3a_9a_3b_2 a_3\beta_2\sigma_3 \sigma_2 a_5\overline{\gamma} \overline{a}_4\oa_5\gamma\sigma a_6 \nonumber \\ &=& a_3\beta_1 a_3b_1a_2a_1a_3b_1\sigma_1a_7a_3\beta_3a_9a_3b_2 a_3\beta_2\sigma_3 \sigma_2(\oa_4)\sigma a_6 \nonumber \\ &=& a_3\beta_1 a_3b_1a_2a_1a_3b_1\sigma_1a_7a_3(\oa_4)(\beta_3)a_9a_3b_2 a_3\beta_2\sigma_3 \sigma_2\sigma a_6 \nonumber \\ &=& a_3\beta_1 a_3b_1a_2a_1a_3b_1\sigma_1a_7a_3(\oa_4)(\oa_4b_2a_4)a_9(a_3b_2a_3)\beta_2\sigma_3 \sigma_2\sigma a_6 \nonumber \\ &=& a_3\beta_1 a_3b_1a_2a_1a_3b_1\sigma_1a_7a_3(\oa_4)(b_2a_4\ob_2)a_9(b_2a_3b_2)\beta_2\sigma_3 \sigma_2\sigma a_6 \nonumber \\ &=& a_3\beta_1 a_3b_1a_2a_1a_3b_1\sigma_1a_7a_3\beta_3\beta_4a_3b_2\beta_2\sigma_3 \sigma_2\sigma a_6
\end{array}
\end{equation}
where $\beta_1 =\oa_1\oa_2b_1a_2a_1$, $\beta_2 = a_5a_4b_2\oa_4\oa_5$, $\beta_3 = \oa_4b_2a_4 = b_2a_4\ob_2$ and $\beta_4 = \ob_2a_9b_2$.
\begin{figure}[hbt]
\begin{center}
    \includegraphics[width=18cm]{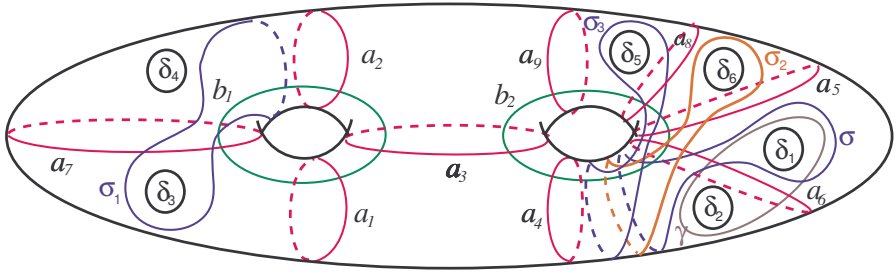}
  \caption{$\Sigma_{2,6}$}
  \label{f9}
\end{center}
 \end{figure}
\section{Final Remarks}
\par  Here, we would like to discuss the total number of disjoint sections that $\CP\#13\CPb \rightarrow S^{2}$ can admit. Let us observe the following:
\begin{lemma}A genus $2$ Lefschetz fibration $\CP\#13\CPb \rightarrow S^{2}$ may admit at most $12$ disjoint sections.
\end{lemma} 
\begin{proof} Suppose $\CP\#13\CPb \rightarrow S^{2}$ admits $13$ disjoint sections. Each section is a sphere with self intersection $-1$. Furthermore,
each section intersects a regular fiber, a genus $2$ surface $\Sigma_{2}$ with self intersection $0$, at one point. Now, by blowing down all $-1$ spheres, we get a genus $2$ surface $\widetilde{\Sigma}_{2}$ with self intersection $13$ which can not exist in a manifold with second homology $\mathbb{Z}$.
\end{proof}
In Section 4, we find relations giving $n$ disjoint sections for genus $2$ Lefschetz fibration $\CP\#13\CPb \rightarrow S^{2}$, for $n =1, \ldots,6$. We want to remark that the technique we applied in Section 4 stops at $n = 6$. However, we can find relations in the corresponding mapping class group that give $n = 7, 8$ disjoint sections for genus $2$ Lefschetz fibration $\CP\#13\CPb \rightarrow S^{2}$ using results given in \cite{KO}. Below we indicate how to derive these relations. This method does not go further either. It remains unknown whether or not there are more than eight sections. 
\par Seven-holed torus relation given in \cite{KO} sits in $\Sigma_{2,7}$. See Figure ~ \ref{f10}. We identify the boundary curves $(\delta_1, \delta_2,\delta_3, \delta_4, \delta_5, \delta_6, \delta_7)$ in $\Sigma_{1,7}$ in Figure ~ \ref{f10} with the curves $(\delta_6, \delta_5, a_2, a_1, \delta_2, \delta_1, \delta_7)$ in $\Sigma_{2,7}$. Seven-holed torus relation gives 
\begin{equation*}
a_1 a_2\delta_1\delta_2\delta_5\delta_6\delta_7 = a_3a_4a_9b_2\sigma_5a_{10}\beta_5\sigma_3\sigma_6a_5\beta_3\sigma_4
\end{equation*}
where $\beta_3 = a_3b_2\overline{a}_3$, $\beta_5 = a_6b_2\overline{a}_6$.
\begin{figure}[h!]
 \includegraphics[width=18cm]{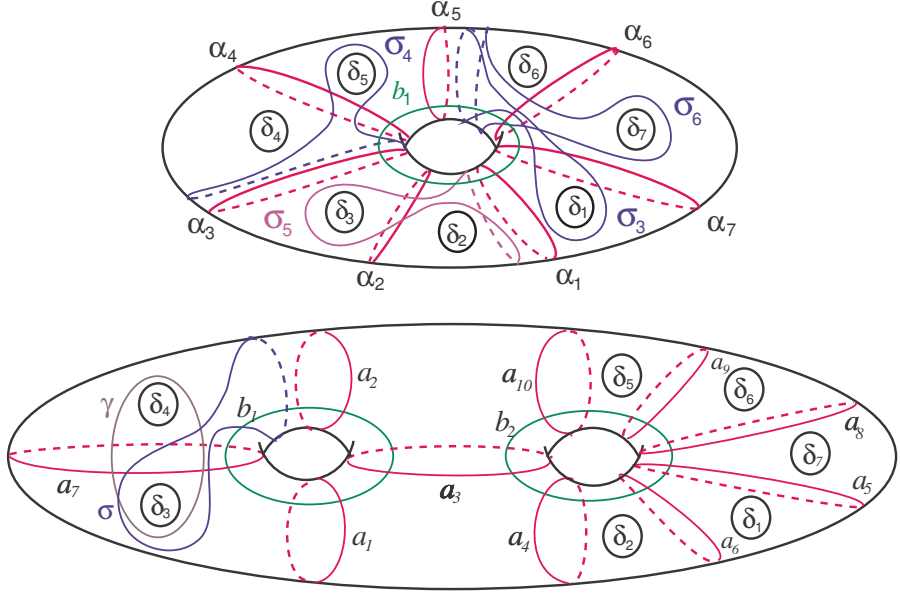}
  \caption{Seven-holed torus relation: $\delta_1\delta_2\delta_3\delta_4\delta_5\delta_6\delta_7 = \a_3\a_4\a_1b_1\sigma_5\a_2\beta_5\sigma_3\sigma_6\a_6\beta_3\sigma_4$ where $\beta_3 = \a_3b_1\overline{\a}_3$, $\beta_5 = \a_5b_1\overline{\a}_5$ and $\Sigma_{2,7}$. Identify $(\delta_1, \delta_2,\delta_3, \delta_4, \delta_5, \delta_6, \delta_7)$ in $\Sigma_{1,7}$ with $(\delta_6, \delta_5, a_2, a_1, \delta_2, \delta_1, \delta_7)$ in $\Sigma_{2,7}$.}
  \label{f10}
 \end{figure}

Next, we combine this relation with the lantern relation $a_1 a_2\delta_3\delta_4= \gamma\sigma a_7$ for the sphere with four boundary curves $\delta_3, \delta_4, a_1, a_2$ in $\Sigma_{2,7}$. Notice the star relation $a_4a_{10}\gamma = (a_1a_2a_3b_1)^3$ for torus with three boundary curves $\gamma, a_4, a_{10}$. We replace $\gamma = \oa_4\oa_{10}(a_1a_2a_3b_1)^3$ in the lantern relation. We simplify the equation and we write the product of right handed Dehn twists along the boundary curves $\delta_1,\ldots, \delta_7$ as a product of twenty right handed Dehn twists along non-boundary parallel simple closed curves on $\Sigma_{2,7}$. 
\begin{equation}
\begin{array}{ccl}
\delta_1\delta_2\delta_3\delta_4\delta_5\delta_6\delta_7 &=& a_3a_4a_9b_2\sigma_5a_{10}\beta_5\sigma_3\sigma_6a_5\beta_3\sigma_4\oa_1\oa_2\oa_1\oa_2\gamma\sigma a_7 \nonumber \\ &=& a_3a_9b_2\sigma_5a_{10}\beta_5\sigma_3\sigma_6a_5\beta_3\sigma_4(a_4)\oa_1\oa_2\oa_1\oa_2\oa_4\oa_{10}(a_1a_2a_3b_1)^3\sigma a_7 \nonumber \\ 
&=& a_3a_9b_2\sigma_5a_{10}\beta_5\sigma_3\sigma_6a_5\beta_3\sigma_4\oa_1\oa_2a_3b_1(a_1a_2a_3b_1)^2(\oa_{10})\sigma a_7 \nonumber \\ 
&=& a_3a_9(\oa_{10})b_2\sigma_5a_{10}\beta_5\sigma_3\sigma_6a_5\beta_3\sigma_4a_3(\oa_1\oa_2b_1a_1a_2)a_3b_1(a_1a_2a_3b_1)\sigma a_7 \nonumber \\
&=& a_3a_9(\oa_{10}b_2a_{10})(\oa_{10}\sigma_5a_{10})\beta_5\sigma_3\sigma_6a_5\beta_3\sigma_4a_3\tilde{\beta}a_3b_1(a_1a_2a_3b_1)\sigma a_7 \nonumber \\ &=& a_3a_9\tilde{\beta_1}\tilde{\beta_2}\beta_5\sigma_3\sigma_6a_5\beta_3\sigma_4a_3\tilde{\beta}a_3b_1(a_1a_2a_3b_1)\sigma a_7
\end{array}
\end{equation}
where $\tilde{\beta_1} =\oa_{10}b_2a_{10}$, $\tilde{\beta_2} = \oa_{10}\sigma_5a_{10}$, $\beta_5 = a_6b_2\overline{a}_6$, $\beta_3  =  a_3b_2\overline{a}_3$ and $\tilde{\beta} = \oa_1\oa_2b_1a_1a_2$. Note that the simple closed curves $\sigma_5$ and $a_{10}$ intersect at $2$ points.
\par Similarly, the eight-holed torus relation given in \cite{KO} sits in $\Sigma_{2,8}$. By applying the same technique, we also get the necessary relation for $n =8$. We identify the boundary curves $(\delta_1, \delta_2,\delta_3, \delta_4, \delta_5, \delta_6, \delta_7, \delta_8)$ in $\Sigma_{1,8}$ in Figure ~ \ref{f12} with the curves $(\delta_1, \delta_8, \delta_7, \delta_6, \delta_5, a_2, a_1, \delta_2)$ in $\Sigma_{2,8}$. Eight-holed torus relation gives 
\begin{equation*}
a_1 a_2\delta_1\delta_2\delta_5\delta_6\delta_7\delta_8 = a_{10}b_2\sigma_5a_{11}\beta_1\sigma_3\sigma_6a_{8}\beta_6\sigma_4\sigma_7a_4
\end{equation*}
where $\beta_1 = a_{5}b_2\overline{a}_{5}$, $\beta_6 = a_3b_2\overline{a}_3$. Then, we combine this relation with the lantern relation $a_1 a_2\delta_3\delta_4= \gamma\sigma a_7$ for the sphere with four boundary curves $\delta_3, \delta_4, a_1, a_2$ in $\Sigma_{2,8}$. Next, using the star relation $a_4a_{11}\gamma = (a_1a_2a_3b_1)^3$ for torus with three boundary curves $\gamma, a_4, a_{11}$, we replace $\gamma = \oa_4\oa_{11}(a_1a_2a_3b_1)^3$ in the lantern relation. We simplify the equation as

\begin{equation}
\begin{array}{ccl}
\delta_1\delta_2\delta_3\delta_4\delta_5\delta_6\delta_7\delta_8 &=& a_{10}b_2\sigma_5a_{11}\beta_1\sigma_3\sigma_6a_{8}\beta_6\sigma_4\sigma_7a_4\oa_1\oa_2\oa_1\oa_2\gamma\sigma a_7 \nonumber \\ &=& a_{10}b_2\sigma_5a_{11}\beta_1\sigma_3\sigma_6a_{8}\beta_6\sigma_4\sigma_7a_4\oa_1\oa_2\oa_1\oa_2\oa_4\oa_{11}(a_1a_2a_3b_1)^3\sigma a_7 \nonumber \\ 
&=& a_{10}b_2\sigma_5a_{11}\beta_1\sigma_3\sigma_6a_{8}\beta_6\sigma_4\sigma_7\oa_1\oa_2a_3b_1(a_1a_2a_3b_1)^2(\oa_{11})\sigma a_7 \nonumber \\ 
&=& a_{10}(\oa_{11})b_2\sigma_5a_{11}\beta_1\sigma_3\sigma_6a_{8}\beta_6\sigma_4\sigma_7a_3(\oa_1\oa_2b_1a_1a_2)a_3b_1(a_1a_2a_3b_1)\sigma a_7 \nonumber \\ &=& a_{10}(\oa_{11}b_2a_{11})(\oa_{11}\sigma_5a_{11})\beta_1\sigma_3\sigma_6a_{8}\beta_6\sigma_4\sigma_7a_3\tilde{\beta}a_3b_1(a_1a_2a_3b_1)\sigma a_7 \nonumber \\
&=& a_{10}\tilde{\beta_1}\tilde{\beta_2}\beta_1\sigma_3\sigma_6a_{8}\beta_6\sigma_4\sigma_7a_3\tilde{\beta}a_3b_1(a_1a_2a_3b_1)\sigma a_7
\end{array}
\end{equation}
where $\beta_1 = a_5b_2\oa_5$, $\beta_6 = a_3b_2\oa_3$, $\tilde{\beta} = \oa_1\oa_2b_1a_1a_2$, $\tilde{\beta_1} = \oa_{11}\b_2a_{11}$ and $\tilde{\beta_2} = \oa_{11}\sigma_5a_{11}$. Note that the simple closed curves $\sigma_5$ and $a_{11}$ intersect at $2$ points.

\begin{figure}[hbt]
\begin{center}
    \includegraphics[width=17cm]{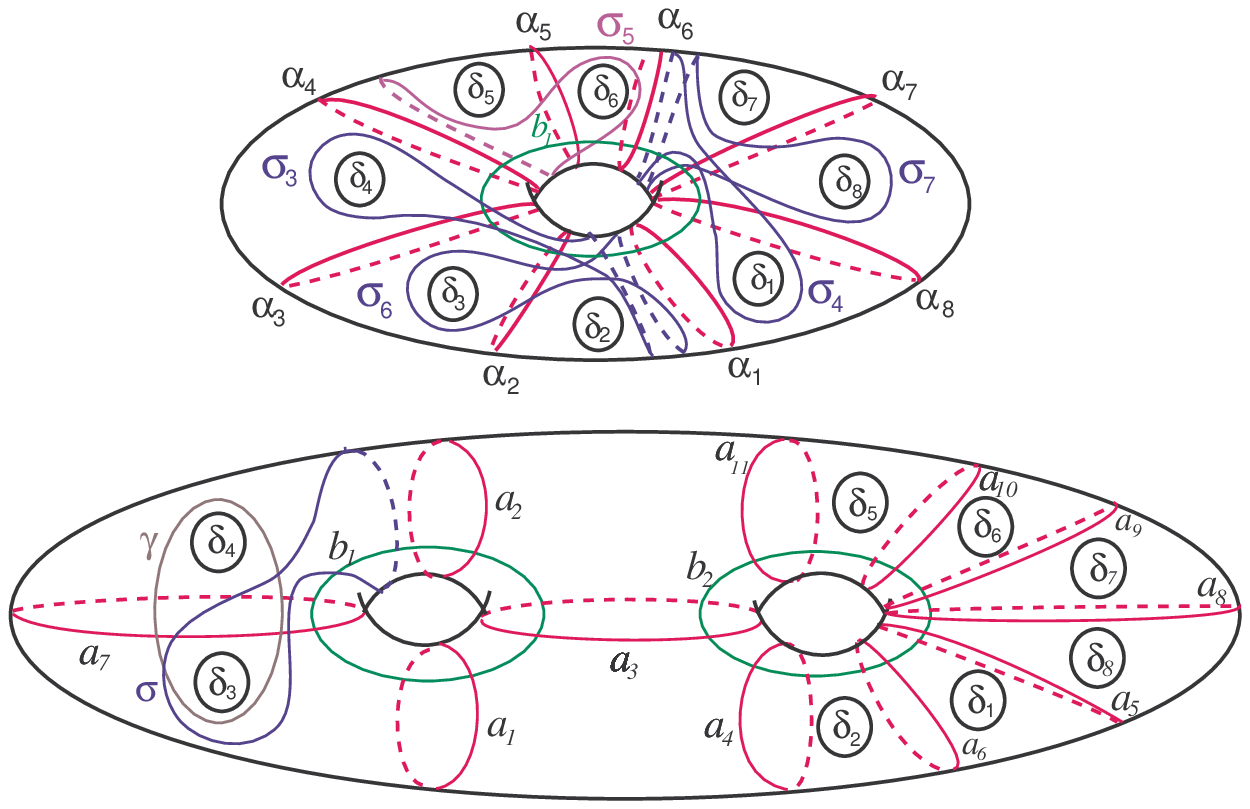}
 \caption{Eight-holed torus relation: $\delta_1\delta_2\delta_3\delta_4\delta_5\delta_6\delta_7\delta_8 = \a_4\a_5\beta_1\sigma_3\sigma_6\a_2\beta_6\sigma_4\sigma_7\a_7\beta_4\sigma_5 = \a_4b_1\sigma_5\a_5\beta_1\sigma_3\sigma_6\a_2\beta_6\sigma_4\sigma_7\a_7$ where $\beta_1=\a_1b_1\overline{\a}_1$, $\beta_4=\a_4b_1\overline{\a}_4$, $\beta_6=\a_6b_1\overline{\a}_6$ and $\Sigma_{2,8}$. Identify $(\delta_1, \delta_2,\delta_3, \delta_4, \delta_5, \delta_6, \delta_7, \delta_8)$ in $\Sigma_{1,8}$ with $(\delta_1, \delta_8, \delta_7, \delta_6, \delta_5, a_2, a_1, \delta_2)$ in $\Sigma_{2,8}$.}
\label{f12}
\end{center}
 \end{figure} 
 
\par By the above lemma, for $n > 12$ there is no relation in the mapping class group $\G_{2,n}$ inducing a genus $2$ Lefschetz fibration $\CP\#13\CPb \rightarrow S^{2}$ with $n$ disjoint sections. We find relations for $n =1, \ldots,8$ giving $n$ disjoint sections for genus $2$ Lefschetz fibration $\CP\#13\CPb \rightarrow S^{2}$. As a consequence by using a result of K. Chakiris we observe: 
 
\begin{corollary}Any genus $2$ holomorphic Lefschetz fibration without separating singular fibers admits a section.
\end{corollary}

 \begin{proof}In \cite{C}, it was shown that any genus $2$ holomorphic Lefschetz fibration without separating singular fibers is obtained by fiber summing the three genus $2$ Lefschetz fibrations given by three relations  
$(c_1c_2c_3c_4c_5^2c_4c_3c_2c_1)^2 = 1, (c_1c_2c_3c_4c_5)^6 = 1, (c_1c_2c_3c_4)^{10} = 1$ in $\G_{2}$ where $c_1, \ldots, c_5$ are simple closed curves shown in Figure ~ \ref{f3}. As mentioned in Section 2, each relation gives us a genus $2$ Lefschetz fibration with total spaces $\CP\#13\CPb, K3\#2\CPb$ and the Horikawa surface $H$, respectively. For Lefschetz fibrations with total spaces $K3\#2\CPb$ and $H$, it is known that they have sections. The relation $(c_1c_2c_3c_4c_5)^6 = \delta_1\delta_2$ in $\G_{2,2}$ gives 2 disjoint sections for Lefschetz fibration $K3\#2\CPb \rightarrow S^{2}$ and the relation $(c_1c_2c_3c_4)^{10} = \delta_1$ in $\G_{2,1}$ gives a section for Lefschetz fibration $H \rightarrow S^{2}$. See Figure ~ \ref{f2}. In this note we provide sections for genus $2$ Lefschetz fibration $\CP\#13\CPb \rightarrow S^{2}$. By sewing the sections during the fiber sum operation, we get a section for any genus $2$ holomorphic Lefschetz fibration without separating singular fibers.
 \end{proof}
\begin{remark} We would like to point out that one may continue and try to write similar relations for $9 \leq n \leq 12$ to see the exact number of disjoint sections that $\CP\#13\CPb \rightarrow S^{2}$ can admit. One can also try to find the exact number of disjoint sections of the genus $2$ Lefschetz fibrations with total spaces $K3\#2\CPb$ and $H$, respectively. Finally, we want to remark that it is still not known whether every genus $g$ Lefschetz fibration over $S^{2}$ admits a section or not. In particular, one may also try to show every genus $g$ holomorphic Lefschetz fibration over $S^{2}$ admits a section or not.
\end{remark} 
\textbf{Acknowledgments.} I am grateful to Anar Akhmedov for his suggestions and his interest in this work. I would like to thank Kenneth L. Baker, John B. Etnyre and Mustafa Korkmaz for helpful comments on the paper. The author is supported in part by TUBITAK-2214.

\end{document}